\documentclass[reqno,a4paper]{amsart}
\usepackage{setspace,amssymb}
\usepackage{ifpdf}
\ifpdf
 \usepackage[hyperindex,pagebackref]{hyperref}%
\else
 \expandafter\ifx\csname dvipdfm\endcsname\relax
 \usepackage[hypertex,hyperindex,pagebackref]{hyperref}
 \else
 \usepackage[dvipdfm,hyperindex,pagebackref]{hyperref}
 \fi
\fi
\allowdisplaybreaks[4]
\theoremstyle{plain}
\newtheorem{thm}{Theorem}
\theoremstyle{remark}

\DeclareMathOperator{\td}{d}
\newcommand{\bell}{\textup{B}}

\begin{document}

\title[An explicit formula for computing Bell numbers]
{An explicit formula for computing Bell numbers in terms of Lah and Stirling numbers}

\author[F. Qi]{Feng Qi}

\address{College of Mathematics, Inner Mongolia University for Nationalities, Tongliao City, Inner Mongolia Autonomous Region, 028043, China; Department of Mathematics, College of Science, Tianjin Polytechnic University, Tianjin City, 300387, China; Institute of Mathematics, Henan Polytechnic University, Jiaozuo City, Henan Province, 454010, China}
\email{\href{mailto: F. Qi <qifeng618@gmail.com>}{qifeng618@gmail.com}, \href{mailto: F. Qi <qifeng618@hotmail.com>}{qifeng618@hotmail.com}, \href{mailto: F. Qi <qifeng618@qq.com>}{qifeng618@qq.com}}
\urladdr{\url{http://qifeng618.wordpress.com}}

\begin{abstract}
In the paper, the author finds an explicit formula for computing Bell numbers in terms of Lah numbers and Stirling numbers of the second kind.
\end{abstract}



\thanks{This paper was typeset using \AmS-\LaTeX}

\maketitle

In combinatorics, Bell numbers, usually denoted by $\bell_n$ for $n\in\{0\}\cup\mathbb{N}$, count the number of ways a set with $n$ elements can be partitioned into disjoint and non-empty subsets. These numbers have been studied by mathematicians since the $19$th century, and their roots go back to medieval Japan, but they are named after Eric Temple Bell, who wrote about them in the $1930$s.
Every Bell number $\bell_n$ may be generated by
\begin{equation}\label{Bell-generate-function}
e^{e^x-1}=\sum_{k=0}^\infty\frac{\bell_k}{k!}x^k
\end{equation}
or, equivalently, by
\begin{equation}\label{Bell-generate-funct-2nd}
e^{e^{-x}-1}=\sum_{k=0}^\infty(-1)^k\bell_k\frac{x^k}{k!}.
\end{equation}
\par
In combinatorics, Stirling numbers arise in a variety of combinatorics problems. They are introduced in the eighteen century by James Stirling. There are two kinds of Stirling numbers: Stirling numbers of the first and second kinds.
Every Stirling number of the second kind, usually denoted by $S(n,k)$, is the number of ways of partitioning a set of $n$ elements into $k$ nonempty subsets, may be computed by
\begin{equation}\label{S(n-k)-explicit}
S(n,k)=\frac1{k!}\sum_{i=0}^k(-1)^i\binom{k}{i}(k-i)^n,
\end{equation}
and may be generated by
\begin{equation}\label{2stirling-gen-funct-exp}
\frac{(e^x-1)^k}{k!}=\sum_{n=k}^\infty S(n,k)\frac{x^n}{n!}, \quad k\in\{0\}\cup\mathbb{N}.
\end{equation}
\par
In combinatorics, Lah numbers, discovered by Ivo Lah in 1955 and usually denoted by $L(n,k)$, count the number of ways a set of $n$ elements can be partitioned into $k$ nonempty linearly ordered subsets and have an explicit formula
\begin{equation}\label{a-i-k-dfn}
L(n,k)=\binom{n-1}{k-1}\frac{n!}{k!}.
\end{equation}
Lah numbers $L(n,k)$ may also be interpreted as coefficients expressing rising factorials $(x)_n$ in terms of falling factorials $\langle x\rangle_n$, where
\begin{equation}
(x)_n=
\begin{cases}
x(x+1)(x+2)\dotsm(x+n-1), & n\ge1,\\
1, & n=0
\end{cases}
\end{equation}
and
\begin{equation}
\langle x\rangle_n=
\begin{cases}
x(x-1)(x-2)\dotsm(x-n+1), & n\ge1,\\
1,& n=0.
\end{cases}
\end{equation}
\par
In~\cite[Theorem~2]{exp-reciprocal-cm.tex} and its formally published paper~\cite[Theorem~2.2]{exp-reciprocal-cm-IJOPCM.tex}, the following explicit formula for computing the $n$-th derivative of the exponential function $e^{\pm1/t}$ was inductively obtained:
\begin{equation}\label{exp-frac1x-expans}
\bigl(e^{\pm1/t}\bigr)^{(n)}
=(-1)^n{e^{\pm1/t}}\sum_{k=1}^{n}(\pm1)^{k}L(n,k)\frac1{t^{n+k}}.
\end{equation}
The formula~\eqref{exp-frac1x-expans} have been applied in~\cite{notes-Stirl-No-JNT-rev.tex, Bessel-ineq-Dgree-CM.tex, QiBerg.tex, simp-exp-degree-revised.tex}.
\par
In combinatorics or number theory, it is common knowledge that Bell numbers $\bell_n$ may be computed in terms of Stirling numbers of the second kind $S(n,k)$ by
\begin{equation}\label{Bell-Stirling-eq}
\bell_n=\sum_{k=1}^nS(n,k).
\end{equation}
\par
In this paper, we will find a new explicit formula for computing Bell numbers $\bell_n$ in terms of Lah numbers $L(n,k)$ and Stirling numbers of the second kind $S(n,k)$.

\begin{thm}\label{Bell-Stirling-Lah-thm}
For $n\in\mathbb{N}$, Bell numbers $\bell_n$ may be computed in terms of Lah numbers $L(n,k)$ and Stirling numbers of the second kind $S(n,k)$ by
\begin{equation}\label{Bell-Stirling-Lah-eq}
\bell_n=\sum_{k=1}^n(-1)^{n-k}\Biggl[\sum_{\ell=1}^{k}L(k,\ell)\Biggr]S(n,k).
\end{equation}
\end{thm}

\begin{proof}
In combinatorics, Bell polynomials of the second kind, or say, partial Bell polynomials, denoted by $\bell_{n,k}(x_1,x_2,\dotsc,x_{n-k+1})$ for $n\ge k\ge1$, are defined by
\begin{equation}
\bell_{n,k}(x_1,x_2,\dotsc,x_{n-k+1})=\sum_{\substack{1\le i\le n,\ell_i\in\mathbb{N}\\ \sum_{i=1}^ni\ell_i=n\\ \sum_{i=1}^n\ell_i=k}}\frac{n!}{\prod_{i=1}^{n-k+1}\ell_i!} \prod_{i=1}^{n-k+1}\Bigl(\frac{x_i}{i!}\Bigr)^{\ell_i}.
\end{equation}
See~\cite[p.~134, Theorem~A]{Comtet-Combinatorics-74}.
The famous Fa\`a di Bruno formula may be described in terms of Bell polynomials of the second kind $\bell_{n,k}(x_1,x_2,\dotsc,x_{n-k+1})$ by
\begin{equation}\label{Bruno-Bell-Polynomial}
\frac{\td^n}{\td t^n}f\circ h(t)=\sum_{k=1}^nf^{(k)}(h(t)) \bell_{n,k}\bigl(h'(t),h''(t),\dotsc,h^{(n-k+1)}(t)\bigr).
\end{equation}
See~\cite[p.~139, Theorem~C]{Comtet-Combinatorics-74}. Taking $f(u)=e^{1/u}$ and $h(x)=e^x$ in~\eqref{Bruno-Bell-Polynomial} and making use of~\eqref{exp-frac1x-expans} give
\begin{align*}
\frac{\td^ne^{e^{-x}}}{\td x^n}&=\sum_{k=1}^n\frac{\td^ke^{1/u}}{\td u^k}
\bell_{n,k}(\overbrace{e^x,e^x,\dotsc,e^x}^{n-k+1})\\
&=\sum_{k=1}^n(-1)^k{e^{1/u}}\sum_{\ell=1}^{k}L(k,\ell)\frac1{u^{k+\ell}} \bell_{n,k}(\overbrace{e^x,e^x,\dotsc,e^x}^{n-k+1})\\
&=e^{e^{-x}}\sum_{k=1}^n(-1)^k\sum_{\ell=1}^{k}L(k,\ell)\frac1{e^{(k+\ell)x}} \bell_{n,k}(\overbrace{e^x,e^x,\dotsc,e^x}^{n-k+1}).
\end{align*}
Further by virtue of
\begin{equation}\label{Bell(n-k)}
\bell_{n,k}\bigl(abx_1,ab^2x_2,\dotsc,ab^{n-k+1}x_{n-k+1}\bigr) =a^kb^n\bell_{n,k}(x_1,x_n,\dotsc,x_{n-k+1})
\end{equation}
and
\begin{equation}\label{Bell-stirling}
\bell_{n,k}\bigl(\overbrace{1,1,\dotsc,1}^{n-k+1}\bigr)=S(n,k)
\end{equation}
listed in~\cite[p.~135]{Comtet-Combinatorics-74}, where $a$ and $b$ are complex numbers, we obtain
\begin{align*}
\frac{\td^ne^{e^{-x}}}{\td x^n}
&=e^{e^{-x}}\sum_{k=1}^n(-1)^k\sum_{\ell=1}^{k}L(k,\ell)\frac1{e^{(k+\ell)x}} e^{kx} \bell_{n,k}(\overbrace{1,1,\dotsc,1}^{n-k+1})\\
&=e^{e^{-x}}\sum_{k=1}^n(-1)^k\sum_{\ell=1}^{k}L(k,\ell)\frac1{e^{\ell x}}S(n,k).
\end{align*}
Comparing this with the $n$-th derivative of the generating function~\eqref{Bell-generate-funct-2nd}
\begin{equation}
\frac{\td^ne^{e^{-x}-1}}{\td x^n}=\sum_{k=n}^\infty(-1)^k\bell_k\frac{x^{k-n}}{(k-n)!}
\end{equation}
yields
\begin{equation*}
e\sum_{k=n}^\infty(-1)^k\bell_k\frac{x^{k-n}}{(k-n)!}
=e^{e^{-x}}\sum_{k=1}^n(-1)^k\sum_{\ell=1}^{k}L(k,\ell)\frac1{e^{\ell x}}S(n,k).
\end{equation*}
Letting $x\to0$ in the above equation reveals
\begin{equation*}
(-1)^ne\bell_n=e\sum_{k=1}^n(-1)^k\sum_{\ell=1}^{k}L(k,\ell)S(n,k)
\end{equation*}
which may be rearranged as~\eqref{Bell-Stirling-Lah-eq}.
The proof of Theorem~\ref{Bell-Stirling-Lah-thm} is complete.
\end{proof}

\end{document}